\documentclass[12pt]{article}

\usepackage{amsmath,amsthm,amsfonts,latexsym,amscd,amssymb,}
\def\qed{{\hbadness=10000\hfill\ \vbox{\hrule height.09ex
     \hbox{\vrule width.09ex height1.55ex depth.2ex \kern1.8ex
     \vrule width.09ex height1.55ex depth.2ex}\hrule height.09ex}\break
     \bigskip}}

\newtheorem{thm}{Theorem}[section]
\newtheorem{prop}[thm]{Proposition}
\newtheorem{cor}[thm]{Corollary}

\numberwithin{equation}{section}

\begin{document}

\title{A Study of Groups through Transversals}

\author{Vivek Kumar Jain*
~ and Vipul Kakkar**\\
*Department of Mathematics, Central University of Bihar \\
Patna (India) 800014\\
**School of Mathematics, Harish-Chandra Research Institute\\
Allahabad (India) 211019\\
Email: jaijinenedra@gmail.com; vplkakkar@gmail.com}

\date{}
\maketitle


\begin{abstract}
In this note, a necessary and sufficient condition for the normalizer of a core-free subgroup $H$ of a finite group $G$ to be normal in $G$ is obtained. Also, a known result of finite groups is obtained through transversal. \end{abstract}



\section{Introduction}
Transversals play an important role in characterizing the group and the embedding of subgroup in the group. In \cite{tm}, Tarski monsters has been characterized with the help of transversals.  In \cite{p}, Lal shows that if all subgroups of a finitely generated solvable group are stable (any right transversals of a subgroup have isomorphic group torsion), then the group is nilpotent. In \cite{ict2}, it has been shown that if the isomorphism classes of transversals of a subgroup in a finite group is $3$, then the subgroup itself has index $3$. Converse of this fact is also true.


Let $G$ be a group and $H$ a proper subgroup of $G$. A {\textit{right transversal}} is a subset of $G$ obtained by selecting one and only one element from each right coset of $H$ in $G$ and identity from the coset $H$. Now we will call it {\textit {transversal}} in place of right transversal. Suppose that $S$ is a transversal of $H$ in $G$. We define an operation $\circ $ on $S$ as follows: for $x,y \in S$, $\{x \circ y \}:=S \cap Hxy$.
It is easy to check that $(S, \circ )$ is a  right loop, that is the equation of the type $X \circ a=b$, $X$ is unknown and $a,b \in S$ has unique solution in $S$ and $(S, \circ)$ has both sided identity.
 In \cite{rltr}, it has been shown that for each right loop there exists a pair $(G,H)$ such that $H$ is a core-free subgroup of the group $G$ and the given right loop can be identified with a tansversal of $H$ in $G$. Not all transversals of a subgroup generate the group. But it is proved by Cameron in \cite{pjc}, that if a subgroup of a finite group is core-free, then always there exists a transversal which generates the whole group. We call such a transversal as {\it generating transversal}.

Let $(S, \circ)$ be a right loop (identity denoted by $1$). Let $x,y,z \in S$. Define a map $f^S(y,z)$ from $S$ to $S$ as follows: $f^S(y,z)(x)$ is the unique solution of the equation $X \circ (y \circ z )=(x \circ y ) \circ z$, where $X$ is unknown. It is easy to verify that $f^S(y,z)$ is a bijective map. For a set $X$, let $Sym(X)$ denote the symmetric group on $X$. We denote by $G_S$ the subgroup of $Sym(S)$, generated by the set $\{f^S(y,z) \mid y,z \in S \}$. This group is called {\textit{group torsion}} of $S$ \cite[Definition 3.1, p. 75]{rltr}. It measures the deviation of a right loop from being a group. Our convention for the product in
the symmetric group $Sym(S)$ is given as $(rs)(x) =s(r(x))$ for $r, s \in Sym(S)$ and $x \in S$. Further, the group $G_S$  acts on $S$ through the action $\theta^S$ defined as: for $x \in S$ and $h \in G_S$; $x \theta^S h := h(x)$. Also note that right multiplication by an element of $S$ gives a bijective map from $S$ to $S$, that is an element of $Sym(S)$. 
The subgroup generated by this type of elements in $Sym(S)$ is denoted by $G_SS$ because $G_S$ is a subgroup of it and the right multiplication map associated with the elements of $S$ form a transversal of $G_S$ in $G_SS$. Note that if $H$ is a core-free subgroup of a group $G$ and $S$ is a generating transversal of $H$ in $G$, then $G\cong G_SS$ such that $H \cong G_S$ \cite[Proposition 3.10, p. 80]{rltr}.

A non-empty subset $T$ of right loop $S$ is called a \textit{right subloop} of $S$, if it is right loop with respect to induced binary operation on $T$ (see \cite[Definition 2.1, p. 2683]{rps}). An equivalence relation $R$ on a right loop $S$ is called a congruence in $S$, if it is a sub right loop of $S \times S$. 
Also an \textit{invariant right subloop} of a right loop $S$ is precisely the equivalence class of the identity of a congruence in $S$ (\cite[Definition 2.8, p. 2689]{rps}).
It is observed in the proof of \cite[Proposition 2.10, p. 2690]{rps} that if $T$ is an invariant right subloop of $S$, then the set $S/T=\{T \circ x|x \in S\}$ becomes right loop called as \textit{quotient of S mod T}. Let $R$ be the congruence associated to an invariant right subloop $T$ of $S$. Then we also denote $S/T$ by $S/R$. 

\section{Some Results Through Transversals}

The group $G_SS$ has the natural action $\star$ on $S$. Consider the set $\sigma(S)=\{(x,y) \in S \times S|stab(G_SS,x)=stab(G_SS,y)\}$, where for a pemutation group $G$ on a set $X$, $stab(G,u)$ denotes the stabilizer of $u \in X$ in $G$. One can check that $\sigma(S)$ is an equivalence relation on $S$. This relation is called as \textit{stability relation} on $S$. One also observes that if $(x,y) \in \sigma(S)$, then $stab(G_SS,x \star p)=p~ stab(G_SS,x)~ p^{-1}=p~ stab(G_SS,y)~ p^{-1}=stab(G_SS,y \star p)$ for all $p \in G_SS$. Consider the equivalence class $\sigma(S)_1$ of $1 \in S$ of a right loop $S$. For $u \in S$, let $R_u(x)=x \circ u$ for all $x \in S$. Let $x,y,z \in S$. Write equation $(C6)$ of \cite[Definition 2.1, p. 71]{rltr} as $f^S(y,z)(x)=(R_yR_zR_{y \circ z}^{-1})(x)$. Recall that our convention for the product in the symmetric group $Sym(S)$ is given as $(rs)(x)=s(r(x))$ for $r,s \in Sym(S)$ and $x \in S$. Which means that $G_S=\langle R_yR_zR_{y \circ z}^{-1} |y,z \in S\rangle$. Let $x \in \sigma(S)_1$. Then  $stab(G_SS,x)=G_S$. This means that $R_yR_zR_{y \circ z}^{-1}(x)=x$ for all $y,z \in S$. This implies that $\sigma(S)_1=\{x \in S|x \circ (y \circ z)=(x \circ y)\circ z, \text{for~all~} y,z \in S\}$. One can observe that $\sigma(S)_1$ is a right subloop of $S$ which is indeed a group.

Let $x \in \sigma(S)_1$ and $f^S(y,z) \in G_S$. By \cite[Lemma 2.11, p. 6]{vkj} and the paragraph before \cite[Lemma 2.11, p. 6]{vkj},

\qquad \qquad $x f^S(y,z)=\eta^S_x(f^S(y,z)) x \theta^S f^S(y,z)$, where $\eta^S_x(f^S(y,z)) \in G_S$

\qquad \qquad \qquad \qquad \ $=\eta^S_x(f^S(y,z))f^S(y,z)(x)$

\qquad \qquad \qquad \qquad \ $=\eta^S_x(f^S(y,z))x$

This implies that $x \in N_{G_SS}(G_S)$, where $N_{G}(H)$ denotes the normalizer of $H$ in $G$. Thus, we have proved following:

\begin{prop}\label{p1} Let $S$ be a right loop and $\sigma(S)$ be the stability relation. Then
\[\sigma(S)_1=\{x \in S|x \circ (y \circ z)=(x \circ y)\circ z, \text{for~all~} y,z \in S\}=N_{G_SS}(G_S) \cap S.\]
\end{prop}
We now have following proposition:
\begin{thm}\label{p2}
Let $S$ be a right loop. Then $N_{G_SS}(N_{G_SS}(G_S))=G_SS$ if and only if $\sigma(S)$ is a congruence on $S$ and $\{(x,x \theta^S f^S(y,z))|x \in S\} \subseteq \sigma(S)$.
\end{thm}
\begin{proof}
By \cite[Lemma 2.12, p. 6]{vkj} and Proposition \ref{p1}, one can observe the 'if' part. We will now observe the 'only if' part.
 
Assume that $N_{G_SS}(N_{G_SS}(G_S))=G_SS$. This implies that $N_{G_SS}(G_S)\trianglelefteq G_SS$. In this case, $\sigma(S)_1$ is the kernel of the homomorphism from right loop $S$ to $G_SS/N_{G_SS}(G_S)$ given by $x \mapsto N_{G_SS}(G_S)x$. Thus, $\sigma(S)$ is a congruence on $S$. Also, by \cite[Lemma 2.9, 2.10, p. 6]{vkj} and Proposition \ref{p1}, $\{(x,x \theta^S f^S(y,z))|x \in S\} \subseteq \sigma(S)$. 

\end{proof}
Let $H$ be a core-free subgroup of a group $G$. Let $S$ be a transversal of $H$ in $G$. Let $u \in S$. In following corollary, $stab(H,u)$ denotes the stabilizer of $H$ at $u$ with respect to the action $\theta$ described before \cite[Lemma 2.1, p. 6]{vkj}.

\begin{cor}
Let $G$ be a finite group and $S$ be a generating transversal of a core-free subgroup $H$ of $G$. Then $N_G(N_G(H))=G$ if and only if $X=\{(x,y) \in S \times S|stab(H,x)=stab(H,y)\}$ is a congruence on the induced right loop structure on $S$ and $\{(x,x \theta h)|x \in S \text{~and~} h \in H\} \subseteq X$.   
\end{cor}  
\begin{thm}\label{2}
Suppose that $S$ is a right loop and $T$ is its invariant sub right loop of order $2$ such that the quotient $S/T$ is a group. Then the group torsion of $S$ will be an elementary abelian $2$-group. 
\end{thm}
\begin{proof}
Suppose that $T=\{1, t\}$. Since $T$ is an invariant sub right loop, so $t$ is fixed by all the elements of $G_S$. 
Consider the equation $(x \circ y ) \circ z = (x \theta^S f^S(y,z)) \circ (y \circ z)$. 
Since $S/T$ has associativity, so  $T \circ ( (x \circ y ) \circ z) = (T \circ x \theta^S f^S(y,z)) \circ ((T \circ y) \circ (T \circ z))$. Since $S/T$ is a group, so by cancellation law we have $T\circ x=T \circ (x \theta^S f^S(y,z))$. This implies $\{x, t \circ x \}=\{x \theta^S f^S(y,z), t \circ x \theta^S f(y,z) \}$. Note that for each $y,z \in S$, an element $x \in S$ is either fixed by $f^S(y,z)$ or $x \theta^S f^S(y,z)=f^S(y,z)(x)=t \circ x$. Suppose that $f^S(y,z)(x)=t \circ x$. Since $f^S(y,z)$ is bijective, so $f^S(y,z)(t \circ x) \neq t \circ x$. Then $f^S(y,z)(t \circ x)= t\circ (t \circ x)=(t \theta^S f^S(y,z)^{-1} \circ t) \circ x=(t \circ t )\circ x =x$. Thus either $x$ is fixed by $f^S(y,z)$ or there is a transposition $(x, t \circ x)$ in the cycle decomposition of $f^S(y,z)$. That is $f^S(y,z)$ can be written as a product of disjoint transpositions. This clearly implies that $G_S= \langle \{f^S(y,z) \mid y,z \in S \} \rangle $ is an elementary abelian 2-group. This proves the theorem.
\end{proof}

Following is a known result (can be proved independently) follows from Theorem \ref{2}. 

\begin{cor}
Let $G$ be a finite group and $H$ a core-free subgroup contained in a normal subgroup $N$ such that the index of $H$ in $N$ is $2$. Then $N$ is an elementary abelian $2$-group.
\end{cor}

\end{document}